\documentclass[11pt]{article}
\usepackage{latexsym,amsfonts,amssymb,amsmath,amsthm}
\usepackage{graphicx}

\usepackage[usenames,dvipsnames]{color}
\usepackage{ulem}

\begin{document}
\setlength{\baselineskip}{16pt}

\parindent 0.5cm
\evensidemargin 0cm \oddsidemargin 0cm \topmargin 0cm \textheight 22.5cm \textwidth 16cm \footskip 2cm \headsep
0cm

\newtheorem{theorem}{Theorem}[section]
\newtheorem{lemma}{Lemma}[section]
\newtheorem{proposition}{Proposition}[section]
\newtheorem{definition}{Definition}[section]
\newtheorem{example}{Example}[section]
\newtheorem{corollary}{Corollary}[section]

\newtheorem{remark}{Remark}[section]

\numberwithin{equation}{section}

\def\p{\partial}
\def\I{\textit}
\def\R{\mathbb R}
\def\C{\mathbb C}
\def\u{\underline}
\def\l{\lambda}
\def\a{\alpha}
\def\O{\Omega}
\def\e{\epsilon}
\def\ls{\lambda^*}
\def\D{\displaystyle}
\def\wyx{ \frac{w(y,t)}{w(x,t)}}
\def\imp{\Rightarrow}
\def\tE{\tilde E}
\def\tX{\tilde X}
\def\tH{\tilde H}
\def\tu{\tilde u}
\def\d{\mathcal D}
\def\aa{\mathcal A}
\def\DH{\mathcal D(\tH)}
\def\bE{\bar E}
\def\bH{\bar H}
\def\M{\mathcal M}
\renewcommand{\labelenumi}{(\arabic{enumi})}

\def\lan{\langle}
\def\ran{\rangle}
\def\da{\downarrow}
\def\ra{\rightarrow}
\def\bs{\backslash}
\def\ol{\overline}

\def\disp{\displaystyle}
\def\undertex#1{$\underline{\hbox{#1}}$}
\def\card{\mathop{\hbox{card}}}
\def\sgn{\mathop{\hbox{sgn}}}
\def\exp{\mathop{\hbox{exp}}}
\def\OFP{(\Omega,{\cal F},\PP)}
\newcommand\JM{Mierczy\'nski}
\newcommand\RR{\ensuremath{\mathbb{R}}}
\newcommand\CC{\ensuremath{\mathbb{C}}}
\newcommand\QQ{\ensuremath{\mathbb{Q}}}
\newcommand\Z{\ensuremath{\mathbb{Z}}}
\newcommand\NN{\ensuremath{\mathbb{N}}}
\newcommand\PP{\ensuremath{\mathbb{P}}}
\newcommand\abs[1]{\ensuremath{\lvert#1\rvert}}
\newcommand\normf[1]{\ensuremath{\lVert#1\rVert_{f}}}
\newcommand\normfRb[1]{\ensuremath{\lVert#1\rVert_{f,R_b}}}
\newcommand\normfRbone[1]{\ensuremath{\lVert#1\rVert_{f, R_{b_1}}}}
\newcommand\normfRbtwo[1]{\ensuremath{\lVert#1\rVert_{f,R_{b_2}}}}
\newcommand\normtwo[1]{\ensuremath{\lVert#1\rVert_{2}}}
\newcommand\norminfty[1]{\ensuremath{\lVert#1\rVert_{\infty}}}

\title{Stability and uniqueness of generalized traveling waves of lattice  Fisher-KPP equations in heterogeneous media}

\author{Feng Cao\\
Department of Mathematics\\
 Nanjing University of Aeronautics and Astronautics\\
  Nanjing, Jiangsu 210016, P. R. China\\
  and\\
 Wenxian Shen\\
Department of Mathematics and Statistics\\
Auburn University\\
Auburn University, AL 36849\\
U.S.A. \\
\\
Dedicated to Professor Min Qian on the occasion of his 90th birthday }

\date{}
\maketitle

\noindent {\bf Abstract.}
In this paper, we investigate the stability and uniqueness of generalized traveling wave solutions of lattice Fisher-KPP equations with
general time and space dependence. We first show the existence, uniqueness, and stability of strictly positive
entire solutions of such equations. Next, we show the stability and uniqueness of generalized traveling waves connecting the unique strictly positive entire solution and the trivial solution zero.
Applying the general stability and uniqueness theorem, we then prove the existence, stability and uniqueness of periodic traveling wave solutions of lattice Fisher-KPP equations in time and space periodic media, and the existence, stability and uniqueness of generalized traveling wave solutions of lattice Fisher-KPP equations in time heterogeneous media. The general stability result established in this paper implies that the generalized traveling waves
obtained in many cases are asymptotically stable under well-fitted perturbation.

\medskip


\medskip


\section{Introduction}

The current paper is to explore the stability and uniqueness  of generalized traveling waves for the following lattice Fisher-KPP equation
\begin{equation}\label{main-eqn}
\dot{u}_{j}(t)=d(t,j+1)\big(u_{j+1}(t)-u_j(t)\big)+d(t,j-1)\big(u_{j-1}(t)-u_j(t)\big)+u_j(t)f(t,j,u_{j}(t)),
\end{equation}
where $j\in\Z$, $\inf_{j\in\Z,t\in\RR} d(t,j)>0$,  and $f(t,j,u)$ is of monostable or Fisher-KPP type. More precisely, we  assume

\medskip

\noindent {\bf (H0)}  {\it $f(t,j,u)$ is locally H\"older continuous in $t\in\R$, Lipschitz continuous in $u\in\R$, and continuously differentiable in $u$ for $u\ge 0$.  Moreover,   $f(t,j,u)=f(t,j,0)$ for $u\le 0$, $f(t,j,u)<0$ for $u\ge M_0$ and some $M_0>0$,   $f_u(t,j,u)<0$ for $u\ge 0$,  and
\begin{equation}
\label{assumption-eq}
\liminf_{t-s\to \infty}\frac{1}{t-s}\int_s^t \inf_{j\in\Z} f(\tau,j,0)d\tau>0.
\end{equation}
}

\medskip

Let
$$
l^\infty=\{u=\{u_i\}_{i\in \Z}:\sup \limits_{i \in\Z}|u_i|<\infty\}
$$
with norm $\|u\|=\|u\|_\infty=\sup_{i\in\Z}|u_i|$, and
 $$
 l^{\infty,+}=\{u\in l^\infty\,|\, u_i\ge 0,\quad \forall\,\, i\in\Z\},\,\, l^{\infty,++}=\{u\in l^\infty\,|\, \inf_{ i\in\Z} u_i> 0\}.
 $$
 By (H0), for any given $u^0\in l^\infty$ and $s\in\R$, \eqref{main-eqn} has a unique (local) solution
$u(t;s,u^0)=\{u_i(t;s,u^0)\}_{i\in\Z}$ with $u(s;s,u^0)=u^0$. Moreover,  if $u^0\in l^{\infty,+}$, then $u(t;s,u^0)=\{u_i(t;s,u^0)\}_{i\in\Z}$ exists for all $t\ge s$ and $u(t;s,u^0)\in l^{\infty,+}$ for all $t\geq s$ (see Lemma \ref{comparison-lm}).

Equation (\ref{main-eqn}) is used to model the  population dynamics of species living in patchy environments in biology and ecology (see, for example, \cite{ShKa97, ShSw90}).
It is the discrete counterpart of the following reaction diffusion equation,
\begin{equation}\label{eqn-con}
u_{t}=d(t,x)u_{xx}+uf(t,x,u).
\end{equation}
Equation \eqref{eqn-con} is widely used to model the population dynamics of species when the movement or internal dispersal of the organisms
occurs between adjacent locations randomly in  spatially continuous media.  Note that for the biological reason, we are only interested in nonnegative solutions of \eqref{main-eqn}. The assumption $f(t,j,u)=f(t,j,0)$
for $u\le 0$ has no effect on nonnegative solutions of \eqref{main-eqn} and is just for convenience.

 One of the central dynamical issues about \eqref{main-eqn} and \eqref{eqn-con}
 is to know how a solution whose initial datum  is strictly positive,  or is a front-like function evolves as time increases.
   For example, it is important to know how a solution $u(t;s,u^0)$ of (\ref{main-eqn}) evolves as $t$ increases, where $u^0$ is strictly positive (that is, $\inf_{i\in\Z}u_i^0>0$), or  $u^0$ is nonnegative and a front-like function (that is,
$$\sup_{i\geq I}u^0_i\ra 0,~~~~\inf_{i\leq -I}u^0_i\ra u_*>0 \quad {\rm as}\quad I\ra \infty.$$
The later is about the front propagation dynamics of  \eqref{main-eqn} and \eqref{eqn-con}, and is strongly related to the  so called  traveling wave solutions  of \eqref{main-eqn}
(resp \eqref{eqn-con}) when $d(t,j)\equiv d$ and $f(t,j,u)\equiv f(u)$ (resp. $d(t,x)\equiv d$ and $f(t,x,u)\equiv f(u)$).

The study of traveling wave solutions of \eqref{eqn-con} traces back to Fisher \cite{Fish37} and Kolmogorov, Petrovsky and Piskunov \cite{KPP37}
 in the special case $d(t,x)=1$ and $f(t,x,u)=1-u$.   Thanks to the pioneering works \cite{Fish37} and \cite{KPP37}, \eqref{main-eqn} and \eqref{eqn-con} with $f$ satisfying
(H0) are called Fisher or KPP type equations in literature.
Since the  works by Fisher (\cite{Fish37}) and Kolmogorov, Petrovsky, Piskunov  (\cite{KPP37}),
traveling wave solutions of Fisher or KPP type evolution equations  in spatially and temporally homogeneous media or spatially and/or temporally
periodic media have been widely studied.
 The reader is referred to \cite{ArWe75, ArWe78,   BeHaNa1, BeHaNa2, BeHaRo, BeNa12, FrGa, HuSh09, Kam76, KoSh, LiZh1, LiZh2, Na09,  NoRuXi, NoXi, NRRZ12, Sa76, Sh10, Sh04,  Uch78, Wei02}, etc.,  for the study of Fisher or KPP type reaction diffusion equations  in homogeneous or periodic media.
The following is a brief review on traveling wave solutions of  Fisher or KPP type lattice equations
in homogeneous or periodic media.

Consider \eqref{main-eqn} in the homogeneous media, that is,  $d(t,j)\equiv d$ and $f(t,j,u)\equiv f(u)$.  By (H0), there is a unique $u^+>0$ such that for any $u^0\in l^\infty$ with $\inf_{j\in\Z} u_j^0>0$,
$$
\lim_{t-s\to\infty} \|u(t;s,u^0)-u^+\|_{l^\infty}=0.
$$
In this case, a solution $\{u_j(t)\}$ of \eqref{main-eqn}  is called a {\it traveling wave solution}  connecting $u=0$ and $u=u^+$
(traveling wave solution for short) if it is an entire solution (i.e. a solution defined for $t\in (-\infty,\infty)$) and there are a constant $c$ and a function $\Phi(\cdot)$  such that
$$
0<u_j(t)=\Phi(j-ct)<u^+\quad \forall\,\, t\in\RR,\,\, j\in \Z,\quad \lim_{z\to -\infty} \Phi(z)=u^+,\,\,\, \lim_{z\to\infty}\Phi(z)=0,
$$
where $c$ and $\Phi(\cdot)$ are  called the {\it wave speed} and {\it wave profile} of the traveling wave solution, respectively.
It is known that there is $c^*>0$ such that \eqref{main-eqn} has a traveling wave solution with speed $c$ if and only if $c\ge c^*$.
The reader is referred to \cite{ChGu02, ChGu03, HuZi94, WuZo97, ZiHaHu93}, etc. for the existence of traveling wave solutions,  and to
\cite{ChGu02, ChGu03, MaZh}, etc. for the uniqueness and stability of traveling wave solutions.

If $d(t,j)$ and $f(t,j,u)$ are periodic in $t$  and $j$ with periods $T\in\RR^+$ and $J\in \Z^+$, respectively, by (H0), there is a unique
positive periodic solution $u^+(t)=\{u_j^+(t)\}$ with $u^+_j(t)=u^+_j(t+T)=u^+_{j+J}(t)$ of \eqref{main-eqn} such that for  any $u^0\in l^\infty$ with $\inf_{j\in\Z} u_j^0>0$,
$$
\lim_{t-s\to\infty} \|u(t;s,u^0)-u^+(t)\|_{l^\infty}=0.
$$
An entire solution $\{u_j(t)\}$ of \eqref{main-eqn} is called a {\it periodic traveling wave solution} or a {\it pulsating wave solution}
connecting $u=0$ and $u=u^+$ if there are a constant $c$ (called {\it wave speed})  and a function $\Phi(\cdot,\cdot,\cdot)$
(called {\it wave profile}) such that
$$
0<u_j(t)=\Phi(j-ct,t,j)<u^+_j(t)\quad \forall \,\,  t\in\RR,\,\, j\in\Z,
$$
$$
 \lim_{x\to -\infty}\Phi(x,t,j)=u^+_j(t),\,\,\,\lim_{x\to\infty}\Phi(x,t,j)=0,
$$
and
$$
\Phi(\cdot,\cdot+T,\cdot)=\Phi(\cdot,\cdot,\cdot+J)=\Phi(\cdot,\cdot,\cdot).
$$
The reader is referred to \cite{GuHa06, HuZi94}, etc. for the existence of periodic traveling wave solutions
and to \cite{GuWu09} for the uniqueness and stability of periodic traveling wave solutions in the case that $d(t,j)$ and
$f(t,j,u)$ are independent of $t$ and periodic in $j$. We note that the existence of periodic traveling wave solutions of \eqref{main-eqn} in the case that
$d(t,j)$ and $f(t,j,u)$ are independent of $j$ and periodic in $t$  follows from the works \cite{LiZh1, Wei82}, and the uniqueness and stability of periodic traveling wave solutions in this case remains open.  We also note that the existence of periodic traveling wave solutions of \eqref{main-eqn} in the case that
$d(t,j)$ and $f(t,j,u)$ are periodic in both $t$ and $j$  follows from the works \cite{LiZh2, Wei02}, and the uniqueness and stability of periodic traveling wave solutions in this case remains open too.

The study of front propagation dynamics of Fisher-KPP type equations with general time and/or space dependence is
more recent,  is attracting more and more attention due to the presence of general time and  space variations in real world problems, but is not much.
To study the front  propagation dynamics of Fisher-KPP type equations with general time and/or space dependence, one first needs
to properly  extend the notion of
traveling wave solutions in the classical sense. Some general extension has been introduced in literature. For example,
in \cite{Sh04}, \cite{Sh11}, notions of random traveling wave solutions and generalized
traveling wave solutions are introduced for random KPP equations and
quite general time dependent KPP equations.
In \cite{BeHa07}, \cite{BeHa12}, a notion of generalized traveling waves is introduced for KPP type equations with general space and
time dependence.

Note that, assuming (H0),  by the similar arguments as those in \cite[Theorem 1.1]{CaSh}, \eqref{main-eqn} has
 a unique strictly positive entire solution
$u^+(t)=\{u^+_j(t)\}$ (i.e $\inf_{j\in\Z, t\in\RR}u^+_j(t)>0$) such that for any $u^0\in l^\infty$  with $\inf_{j\in\Z} u_j^0>0$ and  $s\in\R$,
$$
\lim_{t-s\to\infty} \|u(t;s,u^0)-u^+(t)\|_{l^\infty}=0
$$
(see Proposition \ref{positive-entire-solution-prop}).
An entire solution $\{u_j(t)\}$ of \eqref{main-eqn} is called a {\it generalized traveling wave solution} or {\it transition wave solution}
 connecting $u=0$ and $u=u^+$ if there is a {\it front location function} $X(t)\in\Z$ such that
 $$
 \lim_{j\to -\infty} |u_{j+X(t)}(t)-u_{j+X(t)}^+(t)|=0,\quad \lim_{j\to\infty} u_{j+X(t)}(t)=0
 $$
 uniformly in $t\in\RR$. It is clear that a traveling wave solution of \eqref{main-eqn} in the time and space independent case
 (resp. a periodic traveling wave solution of \eqref{main-eqn} in the time and space periodic case) is a transition wave solution.
 Transition wave solutions for \eqref{eqn-con} are defined similarly.

Quite a few works have been carried out toward the front propagation dynamics of
Fisher-KPP type equations in non-periodic heterogeneous media. For example,
among others, the authors of \cite{NaRo12,  NaRo15, NaRo17} proved the existence of transition waves of \eqref{eqn-con} with general time dependent and space periodic, or  time independent and space almost periodic
KPP nonlinearity.
Zlatos \cite{Zl12} established the existence of transition waves of   spatially inhomogeneous Fisher-KPP reaction diffusion equations { under some specific hypotheses (see (1.2)-(1.5) in \cite{Zl12})}. In \cite{She17}, the stability of transition waves in quite general time and space dependent
Fisher-KPP type reaction diffusion equations is studied.
For spatially discrete KPP equations,
the work  \cite{Sh09} studied spatial spreading speeds of \eqref{main-eqn} with time recurrent KPP nonlinearity $f(t,u)$.
 In the very recent paper \cite{CaSh}, among others, the authors of the current paper established the existence of transition waves in general time dependent Fisher-KPP lattice equations.  However,   there is little study on the stability and uniqueness of transition waves of spatially discrete KPP type equations  with general time and/or space dependence.

 The objective of this paper is to study the stability and uniqueness of transition wave solutions of Fisher-KPP lattice equations in general heterogeneous media and discuss the applications on the existence, stability, and uniqueness of  periodic traveling wave solutions of \eqref{main-eqn} when the coefficients
 are periodic in both $t$ and $j$, and  the applications on the existence, stability, and uniqueness of
  transition wave solutions of \eqref{main-eqn} when
 the coefficients are spatially homogeneous.

 We first establish in Section 2 a general theorem on stability and uniqueness of transition wave solutions of \eqref{main-eqn}
 (see Theorem \ref{main-general-thm}).
 Applying the general stability and uniqueness theorem, we then prove  the existence, stability, and  uniqueness of periodic traveling wave solutions of
 \eqref{main-eqn} when $d(t,j)$ and $f(t,j,u)$ are periodic in $t$ and $j$
  (see Theorem \ref{main-periodic-thm}),  and  the existence, stability and uniqueness of transition wave solutions of \eqref{main-eqn} when $d(t,j)\equiv d(t)$ and $f(t,j,u)\equiv f(t,u)$ (see Theorem \ref{main-time-dependent-thm})  in Section 3 and Section 4, respectively. In the later case, if $d(t)$ and
  $f(t,u)$ are almost periodic in $t$, we also show that the transition waves are almost periodic.
  We will study the existence of transition waves of \eqref{main-eqn} in more general heterogeneous media somewhere else. The general stability and uniqueness theorem established in this paper
 could also be applied to the study of the stability and uniqueness of transition waves in such more general cases.

\section{Stability and uniqueness of transition waves in general heterogeneous media}

In this section, we investigate the stability and uniqueness of transition fronts of \eqref{main-eqn}.

First of all, we have

\begin{proposition}
\label{positive-entire-solution-prop}
Assume (H0). Then there is a unique strictly positive entire solution $u^+(t)=\{u^+_j(t)\}$ such that for any
$u^0\in l^\infty$ with $\inf_{j\in\Z} u^0_j>0$, $u(t;t_0,u^0)$ exists for all $t\ge t_0$, and
$$
\lim_{t\to\infty}\|u(t+t_0;t_0,u^0)-u^+(t+t_0)\|_{l^\infty}=0
$$
uniformly in $t_0\in\RR$.
\end{proposition}

The main results of this section are then stated in the following theorem.

\begin{theorem}
\label{main-general-thm}
Suppose that $u(t)=U(t,\cdot)$ is a transition wave of \eqref{main-eqn} with a front location function $X(t)$ satisfying that
\begin{equation}
\label{cond-eq1}
\forall\,\, \tau>0,\quad \sup_{t,s\in\RR, |t-s|\le \tau} |X(t)-X(s)|<\infty.
\end{equation}
Assume that there are positive continuous functions $\phi(t,j)$ and $\phi_1(t,j)$ such that
\begin{equation}
\label{cond-eq2}
\liminf_{j\to-\infty}\phi(t,j)=\infty,\quad \liminf_{j\to -\infty}\phi_1(t,j)=\infty, \quad \lim_{j\to\infty}\phi(t,j)=0,\quad \lim_{j\to\infty}\phi_1(t,j)=0,
\end{equation}
\begin{equation}
\label{cond-eq3}
\lim_{j\to -\infty}\frac{\phi(t,j+X(t))}{\phi_1(t,j+X(t))}=0,\quad \lim_{j\to\infty}\frac{\phi_1(t,j+X(t))}{\phi(t,j+X(t))}=0
\end{equation}
exponentially, and the second limit in \eqref{cond-eq3} is  uniformly in $t$;
\begin{equation}
\label{cond-eq4}
d^*\phi(t,j)-d_1^* \phi_1(t,j)\le U(t,j)
\le d^*\phi(t,j)+d_1^*\phi_1(t,j)
\end{equation}
for some $d^*,d_1^*>0$ and all $t\in\RR$, $j\in\Z$; and
 for any given $t_0\in\RR$ and $u^0\in l^\infty$ with $u^0_j\ge 0$ and
$$
 u^0_j\ge d\phi(t_0,j)-d_1\phi_1(t_0,j)\quad \Big({\rm resp., }\,\,
u^0_j\le d \phi(t_0,j)+d_1
\phi_1(t_0,j)\Big)
$$
for some $0<d<2d^*$, $d_1\gg 1$, and all $j\in\Z$, there holds
\begin{equation}
\label{cond-eq5}
 u_j(t;t_0,u^0)\ge d\phi(t,j)-d_1 \phi_1(t,j)\quad \Big({\rm resp., }\,\, u_j(t;t_0,u^0) \le  d
\phi(t,j)+d_1 \phi_1(t,j)\Big)
\end{equation}
for all $t\ge t_0$ and $j\in\Z$. Then
the following hold.
\begin{itemize}
\item[(1)] (Stability)  The transition wave $u(t)=U(t,\cdot)$ is asymptotically stable in the sense that
for any $t_0\in\RR$ and $u^0\in l^\infty$ satisfying that
$u^0_j>0$ for all $j\in\Z$ and
\begin{equation}
\label{cond-eq6}
\inf_{j\le j_0}u^0_j>0\quad \forall\,\, j_0\in\Z,\quad \lim_{j\to\infty}\frac{u^0_j}{U(t_0,j)}=1,
\end{equation}
there holds
\begin{equation}
\label{cond-eq7}
\lim_{t\to\infty}\|\frac{u(t+t_0;t_0,u^0)}{U(t+t_0,\cdot)}-1\|_{l^\infty}=0.
\end{equation}

\item[(2)] (Uniqueness) If $u(t)=V(t,\cdot)$ is also a transition wave solution of \eqref{main-eqn} satisfying that
$$
\lim_{j\to\infty} \frac{V(t,j+X(t))}{U(t,j+X(t))}=1
$$
uniformly in $t\in\RR$, then $V(t,\cdot)\equiv U(t,\cdot)$.
\end{itemize}
\end{theorem}

To prove Proposition \ref{positive-entire-solution-prop} and Theorem \ref{main-general-thm}, we first present some lemmas.

 A  function $v(t,j)$ on $[s,t)\times \Z$   is called a {\it super-solution} or {\it sub-solution} of \eqref{main-eqn}  if for any given $j\in\Z$, $v(t,j)$ is  absolutely continuous
in  $t\in [s,T)$,  and
$$ v_t(t,j)\ge d(t,j-1)\big(v(t,j-1)-v(t,j+1)\big)+d(t,j+1)\big(v(t,j+1)-v(t,j)\big)+v(t,j) f(t,j,v(t,j))
$$
 for a.e. $t\in [s,T)$,
or
$$v_t(t,j)\le d(t,j-1)\big(v(t,j-1)-v(t,j)\big)+d(t,j+1)\big(v(t,j+1)-v(t,j)\big)+v(t,j) f(t,j,v(t,j))
$$
 for a.e. $t\in [s,T).$ For given $u,v\in l^\infty$, we define
 $$
 u\le (\ge ) v\quad {\rm if}\quad u_j\le (\ge) v_j\quad \forall\,\, j\in\Z.
 $$

\begin{lemma}[Comparison principle]
\label{comparison-lm}
\begin{itemize}
\item[(1)]
If $u_1(t,j)$ and $u_2(t,j)$ are bounded sub-solution and super-solution of \eqref{main-eqn} on $[s,T)$, respectively, and $u_1(s,\cdot)\leq u_2(s,\cdot)$, then $u_1(t,\cdot)\leq u_2(t,\cdot)$ for $t\in[s,T)$.

\item[(2)] Suppose that $u_1(t,j)$, $u_2(t,j)$ are bounded and satisfy that for any given $j\in\Z$,
 $u_1(t,j)$ and $u_2(t,j)$ are absolutely continuous in $t\in[s,\infty)$, and

\medskip
 \item[]  $\p_t u_2(t,j)-\Big( d(t,j-1)u_2(t,j-1)+d(t,j+1) u_2(t,j+1)-\big(d(t,j-1)+d(t,j+1)\big) u_2(t,j)+u_2(t,j)f(t,j,u_2(t,j))\Big)>\p_t u_1(x,t)-\Big( d(t,j-1)u_1(t,j-1)+d(t,j+1) u_1(t,j+1)-\big(d(t,j-1)+d(t,j+1)\big) u_1(t,j)+u_1(t,j)f(t,j,u_1(t,j))\Big)$

\medskip
    \item[] for a.e. $t>s$.
    Moreover, suppose that $u_2(s,j)\geq u_1(s,j)$. Then $u_2(t,j)>u_1(t,j)$ for $j\in\\Z$, $t>s$.

\item[(3)] If $u^0\in l^{\infty,+}$, then $u(t;s,u^0)$ exists and
 $u(t;s,u^0)\ge 0$ for all $t\ge s$.
\end{itemize}
\end{lemma}

\begin{proof}
It follows from the similar arguments as those in \cite[Proposition 2.1]{CaSh}.
\end{proof}

\begin{lemma}
\label{convergence-lm}
Suppose that $u^{0n},u^0\in l^{\infty,+}$ $(n=1,2,\cdots)$ with $\{\|u^{0n}\|\}$ being  bounded.
If for any $j\in\Z$, $u^{0n}_j\to u^0_j$ as $n\to\infty$,
 then for each $t>0$ and $j\in\Z$, $u_j(s+t;s,u^{0n})- u_j(s+t;s,u^0)\to 0$ as $n\to\infty$ uniformly in $s\in\R$.
\end{lemma}

\begin{proof}
It follows from the similar arguments as those in \cite[Proposition 2.2]{CaSh}.
\end{proof}

For given $u,v\in l^{\infty, +}$,  if
$$\{\alpha>1\, :\, \frac{1}{\alpha}v\leq u\leq \alpha v  \big\}\not =\emptyset,
$$
we define $\rho(u,v)$ by
$$\rho(u,v):=\inf\big\{\ln\alpha:\alpha>1,\frac{1}{\alpha}v\leq u\leq \alpha v  \big\}$$
and call
$\rho(u,v)$ the {\it part metric between $u$ and $v$}.

\begin{lemma}[Part metric]
\label{part-metric-lm}
\begin{itemize}
\item[(1)] For given  $u^0,v^0\in l^{\infty,+}$ with $u^0\not =v^0$, if $\rho(u^0,v^0)$ is well defined, then  $\rho(u(t;s,u_0),u(t;s,v_0))$ is also well defined for
every $t>s$ and $\rho(u(t;s,u_0),u(t;s,v_0))$ is non-increasing in $t$.

\item[(2)] For any $\epsilon>0$, $\sigma>0$, $M>0$,  and $\tau>0$  with $\epsilon<M$ and
$\sigma\le \ln \frac{M}{\epsilon}$, there is $\delta>0$   such that
for any $u^0,v^0\in l^{\infty,++}$ with $\epsilon\le u^0_j\le M$, $\epsilon\le v^0_j\le M$ for $j\in\Z$ and
$\rho(u^0,v^0)\ge\sigma$, there holds
$$
\rho(u(\tau+s;s,u^0),u(\tau+s;s,v^0))\le \rho(u^0,v^0)- \delta\quad\forall\,\, s\in\R.
$$

\item[(3)] Suppose that $u^1(t)$ and $u^2(t)$ are two distinct positive entire solutions of \eqref{main-eqn} and
that  there are $c(t)\in\Z$ and $\mu>0$ such that
$$
\lim_{j\to\infty}\frac{u^i_{j+c(t)}(t)}{e^{-\mu j}}=1
$$
uniformly in $t$ ($i=1,2$), and for any $j_0\in \Z$,
$$
\inf_{j\le j_0,t\in\R}u^i_{j+c(t)}(t)>0
$$
for $i=1,2$. Then for any $\tau>0$ and $T\in\R$, there is $\delta>0$ such that
$$
\rho(u^1(s+\tau),u^2(s+\tau))<\rho(u^1(s),u^2(s))-\delta
$$
for $s\le T$.
\end{itemize}
\end{lemma}

\begin{proof}
It follows from the similar arguments as those in \cite[Proposition 2.3]{CaSh}.
\end{proof}

\begin{proof}[Proof of Proposition \ref{positive-entire-solution-prop}]
It can be proved by the similar arguments as those in \cite[Theorem 1.1]{CaSh}. We give an outline of the proof in the following.

Consider the  linearization of \eqref{main-eqn} at $0$,
\begin{equation}\label{linearization-eq}
\dot v_j=d(t,j-1)\big(v_{j-1}(t)-v_j(t)\big)+d(t,j+1)\big(v_{j+1}(t)-v_j(t)\big)+f(t,j,0)v_j(t).
\end{equation}
 Let $v(t;s,v^0)$ be the solution of \eqref{linearization-eq} with $v(s;s,v^0)=v^0\in l^\infty$. Then
 for any $v^0\in l^\infty$ with $v^0_j\ge 0$,
$$v_j(t;s,v^0)\ge e^{\int^t_s \inf_{j\in\Z}f(\tau,j,0)d\tau} \inf_{j\in\Z}v^0_j.$$
By (H0) we can find $\epsilon_0>0$ and $T>0$ such that
$$
\frac{\int^{s+T}_s\inf_{j\in\Z}f(\tau,j,0)d\tau}{T}>\epsilon_0\quad \forall\,\, s\in\R.
$$
Note that for the above  $\epsilon_0>0$, there is $\delta_0>0$ such that
$$f(t,j,u)\geq \inf_{j\in\Z}f(t,j,0)-\epsilon_0\quad\mbox{ for all }\,t\in\R\,,|u|\leq \delta_0.$$
It then can be proved that for $0<\delta\ll 1$,
$$u(t;s,v^\delta)\geq e^{\int^t_s\inf_{j\in\Z}f(\tau,j,0)d\tau-\epsilon_0(t-s)}v^\delta\quad\mbox{ for }\,s\in\R,\,t\in[s,s+T],
$$
where $v^\delta_j=\delta$ for all $j\in\Z$.
In particular,
$$
u(s+T;s,v^\delta)\geq e^{\int^{s+T}_s\inf_{j\in\Z}f(\tau,j,0)d\tau-\epsilon_0T}v^\delta\ge v^\delta.
$$
By induction, we have
\begin{equation}
\label{positive-solu-eq1}
u(t;s,v^\delta)\geq e^{\int^t_{s+nT}f(\tau,0)d\tau-\epsilon_0(t-s-nT)}v^\delta\quad\mbox{ for }\,s\in\R,\,t\in[s+nT,s+(n+1)T],
\end{equation}
where $n=0,1,2,\cdots$.

By (H0), $f(t,j,u)<0$ for all $t\in\R$, $j\in\Z$ and $u\ge M_0$. Then
\begin{equation}
\label{positive-solu-eq2}
u(t;s,u^M)<u^M\quad\mbox{ for }\,t>s,
\end{equation}
where $M\ge M_0$ and $u^M_j=M$ for all $j\in\Z$.

Let $M\ge M_0$ and $0<\delta\ll 1$ be fixed. Let
$$u^n(t)=u(t;-nT,u^M),\quad t\ge -nT.$$
Then we get
$$u(t;-(n+1)T,v^\delta)<u^{n+1}(t)<u^n(t),\quad t\ge -nT.$$
Let $$u^+(t)=\lim_{n\to\infty}u^n(t).$$ We have that $u^+(t)$ is an entire solution of \eqref{main-eqn}.
 By \eqref{positive-solu-eq1},
 \begin{equation}
 \label{positive-solu-eq3}
 \inf \limits_{j\in\Z, t\in\R}u^+_j(t)>0.
\end{equation}
 Hence $u^+(t)=\{u^+_j(t)\}_{j\in\Z}$ is a strictly positive  entire solution of \eqref{main-eqn}.

By the same arguments as those in \cite[Theorem 1.1]{CaSh}, for any $u^0\in l^\infty$ with $\inf_{j\in\Z}u^0_j>0$,
$$
\lim_{t\to\infty}\|u(t+s;s,u^0)-u^+(t+s)\|_{l^\infty}=0
$$
uniformly in $s\in\RR$. The proposition then follows.
\end{proof}

\begin{proof}[Proof of Theorem \ref{main-general-thm}]
(1) It can be proved by the similar arguments as those in \cite[Theorem 2.2]{She17}. We give an outline of the proof in the following.

First, note that, for given $u^0$ satisfying \eqref{cond-eq6} and given $t_0\in\RR$,
the part metric $\rho(u^0,U(t_0,\cdot))$ is well defined and then
$\rho(u(t,\cdot;t_0,u_0),U(t+t_0,\cdot))$ is well defined for all $t\ge 0$. By Lemma \ref{part-metric-lm},  to prove \eqref{cond-eq7}, it suffices to prove that
 for any $\epsilon>0$,
\begin{equation}
\label{proof-thm1-eq2}
\rho(u(t+t_0,\cdot;t_0,u^0),U(t+t_0,\cdot))<\epsilon\quad \text{for
some}\,\, t>0.
\end{equation}


Second, assume that there is $\epsilon_0>0$ such that
\begin{equation}
\label{eqq-0} \rho(u(t+t_0,\cdot;t_0,u^0),U(t+t_0,\cdot))\ge
\epsilon_0
\end{equation}
for all $t\ge 0$.
Fix a $\tau>0$. We claim that  if \eqref{eqq-0} holds, then there is $\delta>0$ such that
\begin{align}
\label{main-thm-eq1}
\rho(u(\tau+s+t_0,\cdot;t_0,u^0),U(\tau+s+t_0,\cdot))\le
\rho(u(s+t_0,\cdot;t_0,u^0),U(s+t_0,\cdot))-\delta
\end{align}
for all $s\ge 0$.

In fact, for any $0<\epsilon<\frac{\epsilon_0}{4+2\epsilon_0}$, by \eqref{cond-eq2}, \eqref{cond-eq3},  and \eqref{cond-eq6},
there is $d_1\gg d_1^*$ such that
\begin{align*}
d^*(1-\epsilon) \phi(t_0,j)-d_1 \phi_1(t_0,j)\le u^0_j\le d^*(1+\epsilon) \phi(t_0,j)+d_1 \phi_1(t_0,j).
\end{align*}
By \eqref{cond-eq5}, there holds
\begin{align*}
d^*(1-\epsilon) \phi(t,j)-d_1 \phi_1(t,j)\le u_j(t;t_0,u^0)
\le d^* (1+\epsilon)\phi(t,j)+d_1 \phi_1(t,j)\quad \forall\,\, t\ge t_0.
\end{align*}
By the arguments of \cite[(4.5) and (4.6)]{She17}, for any $s\ge 0$, there is $x_s(\ge X(t_0+s))$
 such that
 \begin{equation}
 \label{eqq-00}
 \sup_{s\ge 0, t\in [t_0+s,t_0+s+\tau]}|x_s-X(t)|<\infty
 \end{equation}
 and
\begin{align}
\label{eqq-1}
\frac{1}{1+\epsilon_0/2}U(t,j)\le u_j(t;t_0,u^0)\le (1+\epsilon_0/2)U(t,j)\quad \forall\,\, t\in [t_0+s,t_0+s+\tau],\,\, j\ge x_s.
\end{align}
By the similar arguments of \cite[(4.7)]{She17}, we can prove that
\begin{equation}
\label{eqq-5-0}
\inf_{s\ge 0, t\in[s+t_0,\tau+s+t_0],j\le x_s} U(t,j)>0.
\end{equation}

For given $s\ge 0$,
let
$$
\rho(s+t_0)=\rho(u(s+t_0,\cdot;t_0,u^0),U(s+t_0,\cdot)).
$$
By \eqref{eqq-0},
\begin{equation}
\label{eqq-3-0} \rho(t_0)\ge \rho(s+t_0)\ge \epsilon_0
\end{equation}
and
\begin{align}
\label{eqq-3}
 \frac{1} {e^{\rho(s+t_0)}}U(s+t_0,\cdot)&\le u(s+t_0,\cdot;t_0,u_0)\le  e^{\rho(s+t_0)} U(s+t_0,\cdot).
\end{align}
It follows from \eqref{eqq-3} and Lemma \ref{comparison-lm}  that
$$
u(t+s+t_0,\cdot;t_0,u^0)\le u(t+s+t_0,\cdot;s+t_0, e^{\rho(s+t_0)}U(s+t_0,\cdot))\quad {\rm for}\quad  t\ge 0.
$$

Let
$$
\hat u_j(t)= u_j(t+s+t_0;s+t_0, e^{\rho(s+t_0)}U(s+t_0,\cdot)),
$$
$$\tilde u_j(t)= e^{\rho(s+t_0)} u_j(t+s+t_0;s+t_0,U(s+t_0,\cdot))\big(= e^{\rho(s+t_0)} U(t+s+t_0,j)\big),
$$
and
$$
\bar u_j(t)=\tilde u_j(t)-\hat u_j(t).
$$
Then
\begin{align}
\label{eqq-4}
\p_t\bar u_j(t)&=\big(\mathcal{A}\bar u(t)\big)_j+\tilde u_j(t) f(t+s+t_0,j,U(t+s+t_0,j))-\hat u_j(t) f(t+s+t_0,j,\hat u_j(t))\nonumber\\
&=\big(\mathcal{A}\bar u(t)\big)_j+p(t,j) \bar u_j(t)+b(t,j),
\end{align}
where
$$
(\mathcal{A}\bar u(t)\big)_j=d(t,j-1)\big(\bar u_{j-1}(t)-\bar u_j(t)\big)+d(t,j+1)\big(\bar u_{j+1}(t)-u_j(t)\big),
$$
$$
p(t,j)=f(t+s+t_0,j,\hat u_j(t))+\tilde  u_j(t)\int_0^1 f_u(t+s+t_0,j,r\tilde u_j(t)+(1-r)\hat u_j(t))dr,
$$
and
\begin{align*}
b(t,j)&=\tilde u_j(t) \big[f(t+s+t_0,j,U(t+s+t_0,j))-f(t+s+t_0,j,\tilde u_j(t))\big]\\
&= \tilde  u_j(t)\big[f(t+s+t_0,j,U(t+s+t_0,j))-f(t+s+t_0,j,  e^{\rho(s+t_0)} U(t+s+t_0,j))\big].
\end{align*}
By (H0) and \eqref{eqq-5-0}, there is $b_0>0$ such that for any $s\ge 0$,
\begin{equation}
\label{eqq-5}
\inf_{t\in [s+t_0,\tau+s+t_0],j\le x_s} b(t,x)\ge b_0>0.
\end{equation}

Note that
$$\bar u_{j-1}(t)\ge 0,\quad \bar u_{j+1}(t)\ge 0.
$$
Hence for $j\le x_s$,
$$
\p_t \bar u_j(t)\ge (-d(t,j-1)-d(t,j+1)+ p(t,j))\bar u_j(t)+b_0.
$$
This implies that
$$
\bar u_j(\tau)\ge \int_{s+t_0}^{\tau+s+t_0} e^{\big(-2d_{\sup}+p_{\inf}\big)(\tau+s+t_0-r)}b_0 dr\quad \forall\,\, j\le x_s.
$$
This together with \eqref{eqq-5} implies that there is $\delta_0>0$ such that for any $s\ge 0$,
\begin{equation*}
\bar u_j(\tau)\ge \delta_0\quad \forall \,\, j\le x_s
\end{equation*}
and then
\begin{equation}
\label{eqq-6}
u_j(\tau+s+t_0;t_0,u_0)\le  e^{\rho(s+t_0)}U(t+s+t_0,j)-\delta_0\quad \forall\,\, j\le x_s.
\end{equation}

By \eqref{eqq-1} and \eqref{eqq-6},
then there is $\delta_1>0$ such that
$$
u(\tau+s+t_0;t_0,u^0)\le  e^{\rho(s+t_0)-\delta_1} U(\tau+s+t_0,\cdot)\quad \text{for all}\,\, s\ge 0.
$$
Similarly, we can prove that there is $\delta_2>0$ such that
$$
 \frac{1}{e^{\rho(s+t_0)-\delta_2}}U(\tau+s+t_0,\cdot)\le u(\tau+s+t_0;t_0,u^0)\quad \text{for all}\,\, s\ge 0.
$$
The claim \eqref{main-thm-eq1} then holds for
 $\delta=\min\{\delta_1,\delta_2\}$.

Now we prove that \eqref{main-thm-eq1} gives rise to a contradiction.
In fact, assume \eqref{main-thm-eq1}. Then  we have
$$
\rho(u(n\tau+t_0;t_0,u^0),U(n\tau+t_0,\cdot))\le
\rho(u(t_0;t_0,u^0),U(t_0,\cdot))-n\delta
$$
for all $n\ge 0$. Letting $n\to\infty$, we have  $\rho(u(n\tau+t_0;t_0,u^0),U(n\tau+t_0,\cdot))\to -\infty$, which is
a contradiction. Therefore, \eqref{eqq-0} does not hold and then for any $\epsilon>0$,
$$
\rho(u(t+t_0;t_0,u^0),U(t+t_0,\cdot))<\epsilon\quad \text{for
some}\,\, t>0.
$$
This together with  Lemma \ref{part-metric-lm} implies that
$$
\lim_{t\to\infty}\rho(u(t+t_0,\cdot;t_0,u^0),U(t+t_0,\cdot))=0.
$$
(1) then follows.

(2) Assume that $u_j(t)=V(t,j)$ is also a transition wave and satisfies
that
$$
\lim_{j\to\infty}\frac{V(t,j+X(t))}{U(t,j+X(t))}=1
$$
uniformly in $t$. To prove $V(t,j)\equiv U(t,j)$, it suffices to prove that for any $\epsilon>0$,
$$
\rho(U(t,\cdot),V(t,\cdot))<\epsilon\quad \text{for all}\,\, t\in\RR.
$$

Assume that there are $\epsilon_0>0$ and $t_0\in\RR$ such that
\begin{equation}
\label{general-eq-2}
\rho(U(t_0,\cdot),V(t_0,\cdot))\ge \epsilon_0.
\end{equation}
Then by Lemma \ref{part-metric-lm},
\begin{equation}
\label{general-eq-2-0}
\rho(U(t,\cdot),v(t,\cdot))\ge \epsilon_0\quad \text{for all}\,\, t\le t_0.
\end{equation}
Let $\tau=1$ and $t_n=t_0-n$.
Note that,  for any $\epsilon>0$, there is $J\in\Z^+$ such that
\begin{equation}
\label{general-eq-2-1}
(1-\epsilon) U(t,j+X(t)) <V(t,j+X(t))<(1+\epsilon)U(t,j+X(t))\quad \forall\,\, j\ge J,\,\, t\in\RR
\end{equation}
and
\begin{equation}
\label{general-eq-2-2}
\begin{cases}
U(t,j+X(t))\ge u^+(t,j+X(t))-\epsilon\quad \forall  \, \, j\le -J,\,\, t\in\RR\cr
U(t,j+X(t))\le \epsilon\quad \forall\,\, j\ge J,\,\, t\in\RR.
\end{cases}
\end{equation}
It follows that $X(t)$ is also a front location function of $V(t,j)$. By the similar arguments of \cite[(4.7)]{She17}, we can prove that
\begin{equation}
\label{general-eq-2-3}
\inf_{t\in\RR, j\le J} U(t,j+X(t))>0\quad {\rm and}\quad \inf_{t\in\RR,j\le J} V(t,j+X(t))>0.
\end{equation}
This implies that there is $\rho_0>0$ such that
\begin{equation}
\label{general-eq-2-4}
\rho(U(t,\cdot),V(t,\cdot))\le \rho_0\quad \text{for all}\,\, t\in\RR.
\end{equation}

By the arguments of \eqref{main-thm-eq1} and \eqref{general-eq-2-0}-\eqref{general-eq-2-3}, there is $\delta>0$ such that
$$
\rho(U(t_n+\tau,\cdot),V(t_n+\tau,\cdot))\le \rho(U(t_n,\cdot)-\delta.
$$
This implies that
\begin{equation}
\label{general-eq-2-5}
\rho(U(t_0,\cdot),V(t_0,\cdot))\le \rho(U(t_n,\cdot),V(t_n,\cdot))-n\delta \quad \text{for all}\,\, n\in\mathbb{N}.
\end{equation}
By \eqref{general-eq-2-4} and \eqref{general-eq-2-5},
$$
\rho(U(t_0,\cdot),V(t_0,\cdot))<0,
$$
which is a contradiction. Hence the assumption \eqref{general-eq-2} does not hold and
$\rho(U(t,\cdot),V(t,\cdot))<\epsilon$ for all $\epsilon>0$ and all $t\in\RR$. Therefore,
$U(t,j)\equiv V(t,j)$ and (2) follows.
\end{proof}

\section{Existence, stability and uniqueness of periodic traveling wave solutions}

In this section, we assume that $d(t+T,j)=d(t,j+J)=d(t,j)$ and
$f(t+T,j,u)=f(t,j+J,u)=f(t,j,u)$, and study the existence, stability, and uniqueness of periodic traveling wave solutions of \eqref{main-eqn}.

To state the main results of this section, we first present two propositions.
For any $\mu\in\RR$, consider the following linear equation,
\begin{equation}
\label{periodic-linear-eq1}
\dot v_j(t)=d(t,j-1)\big( e^\mu v_{j-1}(t)-v_j(t)\big)+d(t,j+1)\big(e^{-\mu} v_{j+1}(t)-v_j(t)\big)+f(t,j,0) v_j(t).
\end{equation}
Note that \eqref{periodic-linear-eq1} with $\mu=0$ is the linearized equation of \eqref{main-eqn} at $u\equiv 0$.

\begin{proposition}
\label{principal-eigenvalue-prop}
\begin{itemize}
\item[(1)] For any $\mu\in\RR$, there are $\lambda(\mu)\in\RR$ and $\psi^\mu(t,j)$ with $\psi^\mu(t+T,j)=\psi^\mu(t,j+J)=\psi^\mu(t,j)>0$,
$\|\psi^\mu(0,\cdot)\|_{l^\infty}=1$,  such that
$v_j(t)=e^{\lambda(\mu) t}\psi^\mu(t,j)$ is a solution of \eqref{periodic-linear-eq1}.

\item[(2)] There is $\mu^*>0$  such that
\begin{equation*}
\frac{\lambda(\mu^*)}{\mu^*}=\inf_{\mu>0}\frac{\lambda(\mu)}{\mu},\quad  \frac{\lambda(\mu)}{\mu}>\frac{\lambda(\mu^*)}{\mu^*}\quad {\rm for}\quad 0<\mu<\mu^*.
\end{equation*}
\end{itemize}
\end{proposition}

\begin{proof}
(1) Let
$$
l^\infty_{\rm per}=\{u\in l^\infty\,|\, u_{j+J}=u_j\quad \forall\,\, j\in\Z\}.
$$
For given $u,v\in l^\infty_{\rm per}$, define
$$
u\ll (\gg) v\quad {\rm if}\quad u_j<(>)v_j\quad \forall\,\, j\in\Z.
$$
Let  $\Phi(t,s)$ be the solution operator of \eqref{periodic-linear-eq1}, that is,
$$
\Phi(t,s)v^0=v(t;s,v^0),
$$
where $v(t;s,v^0)$ is the solution of \eqref{periodic-linear-eq1} with $v(s;s,v^0)=v^0\in l^\infty$. Then we have
$$
\Phi(t,s)l^\infty_{\rm per}\subset l^\infty_{\rm per}\quad \forall\,\, t\ge s,
$$
and for any $v^0\in \big(l^{\infty,+}\cap l^\infty_{\rm per}\big)\setminus\{0\}$,
$$
\Phi(t,s)v^0\gg 0\quad \quad \forall t>s.
$$
It is clear that any bounded set $E\subset l^\infty_{\rm per}$, $\Phi(T,0)E$ is relatively compact. Hence by the Krein-Rutman Theorem (see \cite{Hirsch05}),
the spectral radius $r(\Phi(T,0)|_{l_{\rm per}^\infty})$ is an isolated algebraic simple eigenvalue of $\Phi(T,0)|_{l_{\rm per}^\infty}$
with a positive eigenfunction $\psi^*\in l^\infty_{\rm per}$, $\|\psi^*\|_{l^\infty}=1$.
(1) follows with
$$\lambda(\mu)=\frac{\ln r(\Phi(T,0)|_{l_{\rm per}^\infty})}{T}$$
 and
 $$\psi^{\mu}(t,\cdot)=e^{-\lambda(\mu) t}\Phi(t,0)\psi^*.
 $$

(2) Note that $$\lambda(\mu)\ge d_{\min}(e^\mu-1)+d_{\max}(e^{-\mu}-1)+f_{min},$$
where $d_{\min}=\inf_{j\in\Z,\,t\in\R}d(t,j)$, $d_{\max}=\sup_{j\in\Z,\,t\in\R}d(t,j)$, $f_{\min}=\inf_{j\in\Z,\,t\in\R} f(t,j,0)$. We then have
$$\frac{\lambda(\mu)}{\mu}\ge \frac{d_{\min}(e^\mu-1)+d_{\max}(e^{-\mu}-1)+f_{min}}{\mu}\to\infty\quad {\rm as}\quad \mu\to\infty.
 $$
 By $\eqref{assumption-eq}$, we have
 $$\frac{\lambda(\mu)}{\mu}\to\infty\quad {\rm as}\quad  \mu\to 0^+.
 $$
  The conclusion then follows.
\end{proof}

Let $$
c^*=\frac{\lambda(\mu^*)}{\mu^*}.
$$
Then for any $c>c^*$, there is $\mu\in (0,\mu^*)$ such that
$$
c=\frac{\lambda(\mu)}{\mu}.
$$
For given $c>c^*$, let $0<\mu<\mu^{'}<\min\{2 \mu, \mu^*\}$ be such that $c=\frac{\lambda(\mu)}{\mu}$ and $\frac{\lambda(\mu)}{\mu}>\frac{\lambda(\mu^{'})}{\mu^{'}}>c^*$.

Consider the space shifted equations of \eqref{main-eqn},
\begin{equation}\label{space-shifted-eqn}
\dot{u}_{j}(t)=H(i)u_j(t)+u_j(t)f(t,j+i,u_{j}(t)),
\end{equation}
where
$$
 H(i)u_j(t)=d(t,j+i+1)\big(u_{j+1}(t)-u_j(t)\big)+d(t,j+i-1)\big(u_{j-1}(t)-u_j(t)\big), \quad j\in\Z
 $$
 for any $i\in\Z$. Let $u(t,j;u_0,i)$ be the solution of \eqref{space-shifted-eqn} with $u(0,j;u_0,i)=u_0(j)$ for $u_0\in l^\infty$.

For given  $d,d_1>0$, let
$$ \underline{v}(t,j;i,d,d_1)= d e^{-\mu(j-ct)}\psi^\mu(t,j+i)-d_1e^{-\mu^{'}(j-ct)}\psi^{\mu^{'}}(t,j+i).$$
Observe that for given $0<b\ll1$, there are $M>N>0$ such that
\begin{equation}\label{sub-inequ}
{ b\psi^0(t,j+i)}\le\underline{v} (t,j;i,d,d_1)\quad\quad\quad\forall N\le j-ct\le M.\end{equation}
Let $b>0$ and $M>0$ be such that $\eqref{sub-inequ}$ holds, and let
\begin{equation}\label{sub-solution}
\underline{u} (t,j;i,d,d_1,b)=\begin{cases}\max\{{ b\psi^0(t,j+i)},\underline{v} (t,j;i,d,d_1)\},\quad j\le M+ct\cr
\underline{v} (t,j;i,d,d_1),\quad j\ge M+ct.
\end{cases}
\end{equation}

\begin{proposition}
\label{sub-solution-prop}
 Let $0<d\le 2$.
 \begin{itemize}
\item[(1)] For any $i\in\Z$ and  $\frac{d_1}{d}\gg1$, $\underline{v} (t,j;i,d,d_1)$ is a sub-solution of \eqref{space-shifted-eqn}.
\item[(2)] For any $i\in\Z$ and $0<b\ll1$, $u(t,j;i):=b\psi^0(t,j+i)$ is a sub-solution of \eqref{space-shifted-eqn}.
\item[(3)] For  $\frac{d_1}{d}\gg1$ and $0<b\ll1$, $u(t,j;\underline{u} (0,\cdot;i,d,d_1,b),i)\ge\underline{u} (t,j;i,d,d_1,b)$ for $t\ge0$.
\end{itemize}
\end{proposition}
\begin{proof}
(1) First of all, let  $\varphi=de^{-\mu(j-ct)}\psi^\mu(t,j+i)$ and $ \varphi_1= d_1e^{-\mu^{'}(j-ct)}\psi^{\mu^{'}}(t,j+i)$. Let $\bar{M}= d\max\limits_{t\in\R,j\in\Z}\psi^\mu(t,j)$. Let $L>0$ be such that $-f_u(j+i,u)\le L$ for $0\le u\le \bar{M}$. Let $d_0$ be defined by
$$d_0=\max\{\frac{\max\limits_{t\in\R,j\in\Z}\psi^\mu(t,j)}{\min\limits_{t\in\R,j\in\Z}\psi^{\mu^{'}}(t,j)},\frac{L\max\limits_{t\in\R,j\in\Z}[\psi^\mu(t,j)]^2}{(\mu^{'}c-\lambda(\mu^{'}))\min\limits_{t\in\R,j\in\Z}\psi^{\mu^{'}}(t,j)}\}$$

Fix $i\in\Z$. We prove that $\underline{v} (t,j;i,d,d_1)$ is a sub-solution of \eqref{space-shifted-eqn} for $\frac{d_1}{d}\ge d_0$. First, for $(t,j)\in\R\times\Z$ with $\underline{v} (t,j;i,d,d_1)\le 0$,  by (H0), $f(t,j+i,\underline{v} (t,j;i,d,d_1))=f(t,j+i,0)$. Hence
\begin{equation*}
\underline{v}_t-[ H(i)\underline{v}(t,j;i,d,d_1)+\underline{v}(t,j;i,d,d_1)f(t,j+i,\underline{v} (t,j;i,d,d_1))]=-(\mu^{'}c-\lambda(\mu^{'}))\varphi_1\le0.
\end{equation*}

Next, consider $(t,j)\in\R\times\Z$ with $\underline{v} (t,j;i,d,d_1)> 0$. By $\frac{d_1}{d}\ge d_0$, we must have $j-ct\ge0$. Then $ \underline{v}(t,j;i,d,d_1)\le de^{-\mu(j-ct)}\psi^\mu(t,j+i)\le d\psi^\mu(t,j+i)\le \bar{M}$. Note that for $0<y<\bar{M}$,
\begin{align*}
-(\mu^{'}c-\lambda(\mu^{'}))-f_u(t,j+i,y)\frac{(\varphi)^2}{\varphi_1}&\le -(\mu^{'}c-\lambda(\mu^{'}))+L\frac{(\varphi)^2}{\varphi_1}\\&= -(\mu^{'}c-\lambda(\mu^{'}))+ \frac{Ld[\psi^\mu(t,j+i)]^2}{d_1 \psi^{\mu^{'}}(t,j+i)}e^{(\mu^{'}-2\mu)(j-ct)}\\&\le -(\mu^{'}c-\lambda(\mu^{'}))+ \frac{Ld\max\limits_{t\in\R,j\in\Z}[\psi^\mu(t,j)]^2}{d_1\min\limits_{t\in\R,j\in\Z}\psi^{\mu^{'}}(t,j)}\\&\le 0.
\end{align*}
Therefore, for  $(t,j)\in\R\times\Z$ with $\underline{v} (t,j;i,d,d_1)> 0$,
\begin{align*}
&\underline{v}_t-[H(i)\underline{v}(t,j;i,d,d_1)+\underline{v}(t,j;i,d,d_1)f(t,j+i,\underline{v}(t,j;i,d,d_1))]\\=&(\mu c-\lambda(\mu))\varphi-(\mu^{'}c-\lambda(\mu^{'}))\varphi_1+\underline{v}(t,j;i,d,d_1)f(t,j+i,0)-\underline{v}(t,j;i,d,d_1)f(t,j+i,\underline{v})\\=&-(\mu^{'}c-\lambda(\mu^{'}))\varphi_1-f_u(t,j+i,y)(\varphi-\varphi_1)^2\quad\quad\mbox{(for some }y\in(0,\bar{M}) )\\\le&-(\mu^{'}c-\lambda(\mu^{'}))\varphi_1-f_u(t,j+i,y)(\varphi)^2\\=&[-(\mu^{'}c-\lambda(\mu^{'}))-f_u(t,j+i,y)\frac{(\varphi)^2}{\varphi_1}]\varphi_1\\\le&0.
\end{align*}
(1) then follows.

(2)
Fix $i\in\Z$. Observe that
$$ H(i)\psi^0(t,j+i)+f(t,j+i,0)\psi^0(t,j+i)-\psi^0_t(t,j+i)=\lambda(0)\psi^0(t,j+i)\quad\forall\, j\in\Z.$$
Observe also that $\max\limits_{t\in\R,j\in\Z}\lambda(0)\psi^0(t,j+i)>0$ and then
$$\lambda(0)b\psi^0(t,j+i)\ge (f(t,j+i,0)-f(t,j+i,b\psi^0(t,j+i)))b\psi^0(t,j+i)\quad\forall\,0<b\ll1.$$
It then follows that
$$ H(i)b\psi^0(t,j+i)+f(t,j+i,b\psi^0(t,j+i))b\psi^0(t,j+i)-b\psi^0_t(t,j+i)\ge0\quad\forall\, j\in\Z,0<b\ll1. $$
Hence $u(t,j;i):=b\psi^0(t,j+i)$ is a sub-solution of \eqref{space-shifted-eqn} for $0<b\ll1$.

(3)
Let $\tilde{w}(t,j;i)=e^{Ct}(u(t,j;\underline{u}(0,\cdot;i,d,d_1,b),i)-\underline{v}(t,j;i,d,d_1))$, where $C$ is some positive constant to be determined later. Recall that $u(t,j;\underline{u}(0,\cdot;i,d,d_1,b),i)$ is the solution of \eqref{space-shifted-eqn} with $u(0,j;\underline{u}(0,\cdot;i,d,d_1,b),i)=\underline{u}(0,j;i,d,d_1,b)$. Then
$$\tilde{w}_t(t,j;i)\ge  H(i)\tilde{w}(t,j;i)+(C+\tilde{a}(t,j;i))\tilde{w}(t,j;i),$$
where
\begin{align*}\tilde{a}(t,j;i)=&f(t,j+i,u(t,j;\underline{u}(0,\cdot;i,d,d_1,b),i))\\+&\underline{v}(t,j;i,d,d_1)\int_0^1
f_u(t,j+i,\tau(u(t,j;\underline{u}(0,\cdot;i,d,d_1,b),i)-\underline{v}(t,j;i,d,d_1)))d\tau.
\end{align*}
Hence
\begin{equation}\label{difference-ineqn1}
\tilde{w}(t,j;i)\ge \tilde{w}(0,j;i)+\int_0^t [ H(i)\tilde{w}(s,j;i)+(C+\tilde{a}(s,j;i))\tilde{w}(s,j;i)]ds
\end{equation} for all $j\in\Z$. Similarly, let $\bar{w}(t,j;i)=e^{Ct}(u(t,j;\underline{u}(0,\cdot;i,d,d_1,b),i)-b\psi^0(t,j+i))$. Then
\begin{equation}\label{difference-ineqn2}
\bar{w}(t,j;i)\ge \bar{w}(0,j;i)+\int_0^t [ H(i)\bar{w}(s,j;i)+(C+\bar{a}(s,j;i))\bar{w}(s,j;i)]ds
\end{equation} for $j\in\Z$, where
\begin{align*}\bar{a}(t,j;i)=&f(t,j+i,u(t,j;\underline{u}(0,\cdot;i,d,d_1,b),i))\\+&b\psi^0(t,j+i)\int_0^1
f_u(t,j+i,\tau(u(t,j;\underline{u}(0,\cdot;i,d,d_1,b),i)-b\psi^0(t,j+i)))d\tau.
\end{align*}
Let  $w(t,j;i)=e^{Ct}(u(t,j;\underline{u}(0,\cdot;i,d,d_1,b),i)-\underline{u}(t,j;i,d,d_1,b))$. Choose $C>0$ such that
$C+\tilde{a}(s,j;i)>0$ and $C+\bar{a}(s,j;i)>0$. Note that
\begin{equation*}
w(t,j;i)=\begin{cases}\min\{\tilde{w}(t,j;i),\bar{w}(t,j;i)\},\quad j\le M+ct\cr
\tilde{w}(t,j;i),\quad j\ge M+ct.
\end{cases}
\end{equation*}
By \eqref{sub-solution}, \eqref{difference-ineqn1} and \eqref{difference-ineqn2},
\begin{equation*}
\tilde{w}(t,j;i)\ge w(0,j;i)+\int_0^t [ H(i)w(s,j;i)+(C+\tilde{a}(s,j;i))w(s,j;i)]ds
\end{equation*} for $j\in\Z$, and
\begin{equation*}
\bar{w}(t,j;i)\ge w(0,j;i)+\int_0^t [ H(i)w(s,j;i)+(C+\bar{a}(s,j;i))w(s,j;i)]ds
\end{equation*} for $j\le M+ct$. It then follows that
\begin{equation*}
w(t,j;i)\ge w(0,j;i)+\int_0^t [ H(i)w(s,j;i)+(C+\tilde{a}(s,j;i))w(s,j;i)]ds\quad\mbox{for }j\in\Z.
\end{equation*} By the arguments in Lemma \ref{comparison-lm}, we have $w(t,j;i)\ge0$ for $t\ge0$, $j,\,i\in\Z$. Then
$$u(t,j;\underline{u}(0,\cdot;i,d,d_1,b),i)\ge\underline{u}(t,j;i,d,d_1,b)\quad\mbox{for }t\ge0\mbox{ and }j,\,i\in\Z.$$
\end{proof}

Let  $$ \bar{v}(t,j;i,d,d_1)=de^{-\mu(j-ct)}\psi^\mu(t,j+i)+d_1 e^{-\mu^{'}(j-ct)}\psi^{\mu^{'}}(t,j+i)$$ and
\begin{equation}
\label{super-solution}
\bar{u} (t,j;i,d,d_1)=\min\{\bar{v}(t,j;i,d,d_1), u^+_{j+i}(t)\}.
\end{equation}

\begin{proposition}
\label{sup-solution-prop}
 \begin{itemize}
\item[(1)] For any $i\in\Z$,  $d>0$, and $d_1\ge 0$,  $\bar{v}(t,j;i,d,d_1)$ is a super-solution of \eqref{space-shifted-eqn}.
\item[(2)]  $u(t,j;\bar{u}(0,\cdot;i,d,d_1),i)\le\bar{u}(t,j;i,d,d_1)$ for $t\ge0$.
\end{itemize}
\end{proposition}
\begin{proof}
(1)  Let  $\varphi=de^{-\mu(j-ct)}\psi^\mu(t,j+i)$ and $ \varphi_1= d_1e^{-\mu^{'}(j-ct)}\psi^{\mu^{'}}(t,j+i)$.
By direct calculation, we have
\begin{align*}
&\bar{v}_t-[ H(i)\bar{v}(t,j;i,d,d_1)+ \bar{v}(t,j;i,d,d_1)f(t,j+i, \bar{v}(t,j;i,d,d_1))]\\\ge& \bar{v}_t-[H(i)\bar{v}(t,j;i,d,d_1)+\bar{v}(t,j;i,d,d_1)f(t,j+i,0)]\\=& (\mu c-\lambda(\mu))\varphi+(\mu^{'} c-\lambda(\mu^{'}))\varphi_1\\\ge&0.
\end{align*}
(2) By comparison principle,
$$u(t,j;\bar{u}(0,\cdot;i,d,d_1),i)\le\bar{v}(t,j;i,d,d_1)$$
and $$u(t,j;\bar{u}(0,\cdot;i,d,d_1),i)\le u^{+}(t)$$ for $t\ge0$. (2) then follows.
\end{proof}

Let
$$
\phi(t,j)=e^{-\mu(j-ct)}\psi^\mu(t,j),\quad \phi_1(t,j)=e^{-\mu^{'}(j-ct)}\psi^{\mu^{'}}(t,j).
$$

\begin{proposition}
\label{sub-super-solution-prop} Let $0<d\le 2$.
For any $u^0\in l^{\infty,+}$, if
$$
u^0\le {d} \phi(t_0,\cdot)+d_1 \phi_1(t_0,\cdot)\quad ({\rm resp.,} \,\, u^0\ge {d}\phi(t_0,\cdot)-d_1\phi_1(t_0,\cdot)),
$$
then for $\frac{d_1}{d}\gg 1$,
$$
u(t;t_0,u^0)\le {d} \phi(t,\cdot)+d_1 \phi_1(t,\cdot)\quad ({\rm resp.,} \,\, u(t;t_0,u^0)\ge {d} \phi(t,\cdot)-d_1 \phi_1(t,\cdot))
$$
for $t\ge t_0$.
\end{proposition}

\begin{proof}
It follows from Proposition \ref{sub-solution-prop}, Proposition \ref{sup-solution-prop} and
$f_u(t,j,u)<0$ for $u\ge0$.
\end{proof}

We now state the main results of this section.

\begin{theorem}
\label{main-periodic-thm}
Consider \eqref{main-eqn} and assume that $d(t+T,j)=d(t,j+J)=d(t,j)$ and
$f(t+T,j,u)=f(t,j+J,u)=f(t,j,u)$,
\begin{itemize}
\item[(1)] (Existence) For any $c>c^*$, there is a periodic traveling wave solution $u_j(t)=U(t,j)$ with speed $c$ satisfying that
\begin{equation}
\label{periodic-wave-eq1}
\phi(t,j)-d_1^* \phi_1(t,j)\le U(t,j)\le \phi(t,j)+d^*_1 \phi_1(t,j)
\end{equation}
for some $d_1^*>0$.

\item[(2)] (Stability) For any $c>c^*$, $t_0\in\RR$, and $u^0\in l^{\infty,+}$ satisfying
$$
\inf_{j\le j_0} u^0_j>0,\quad \lim_{j\to\infty}\frac{u^0_j}{U(t_0,j)}=1,
$$
there holds
$$
\lim_{t\to\infty} \frac{u_j(t+t_0;t_0,u^0)}{U(t+t_0,j)}=1\quad \text{uniformly in}\quad j\in\Z.
$$

\item[(3)] (Uniqueness) If $u_j(t)=V(j,t)$ is also a periodic traveling wave solution of \eqref{main-eqn} with speed $c$ and satisfying
that
$$
\lim_{j\to\infty} \frac{V(t,j+[ct])}{U(t,j+[ct])}=1\quad \text{uniformly in}\,\, t\in\RR,
$$
then
$$
V(t,j)\equiv U(t,j).
$$
\end{itemize}
\end{theorem}

In order to prove the existence of the periodic traveling wave solution, we consider the following space continuous version of \eqref{space-shifted-eqn},
\begin{equation}\label{space-continuous-eqn}
\partial_tu(t,x)= H(z)u(t,x)+u(t,x)f(t,x+z,u(t,x))\quad\quad x\in\R,z\in\R,
\end{equation}
where
$$ H(z)u(t,x)=d(t,x+z+1)\big(u(t,x+1)-u(t,x)\big)+d(t,x+z-1)\big(u(t,x-1)-u(t,x)\big),
 $$
 and $d(t,x)=d(t,j)$, $f(t,x,u)=f(t,j,u)$ for $j\le x<j+1$. Let $u(t,x;u_0,z)$ be the solution of \eqref{space-continuous-eqn} with $u(0,x;u_0,z)=u_0(x)$ for $u_0\in l^\infty(\R)=\{u=\{u(x)\}_{x\in \R}:\sup \limits_{x \in\R}|u(x)|<\infty\}$.

  Let $u^+_{x}(t)=u^+_j(t)$,  $\psi^{\mu}(t,x)=\psi^{\mu}(t,j)$, and $\psi^{\mu^{'}}(t,x)=\psi^{\mu^{'}}(t,j)$ for $t\in\R$ and $x\in\R$ with $j\le x<j+1$, $j\in\Z$. Let
 $$
 \underbar v(t,x;z,d,d_1)=de^{-\mu (x-ct)}\psi^\mu(t,x+z)-d_1e^{-\mu^{'}(x-ct)}\psi^{\mu^{'}}(t,x+z)\quad {\rm for}\,\, t,x,z\in\R,
 $$
\begin{equation}\label{sub-solution-new}
\underline{u} (t,x;z,d,d_1,b)=\begin{cases}\max\{{ b\psi^0(t,x+z)},\underline{v} (t,x;z,d,d_1)\},\quad x\le M+ct\cr
\underline{v} (t,x;z,d,d_1),\quad x\ge M+ct,
\end{cases}
\end{equation}
 and
   $$ \bar{v}(t,x;z,d,d_1)=de^{-\mu(x-ct)}\psi^\mu(t,x+z)+d_1 e^{-\mu^{'}(x-ct)}\psi^{\mu^{'}}(t,x+z),$$
\begin{equation}
\label{super-solution-new}
\bar{u} (t,x;z,d,d_1)=\min\{\bar{v}(t,x;z,d,d_1), u^+_{x+z}(t)\}.
\end{equation}
 By the similar arguments as those in Propositions \ref{sub-solution-prop} and \ref{sup-solution-prop},
we can also get that,
 for  $0<d\le 2$, $\frac{d_1}{d}\gg1$, and $0<b\ll1$, $u(t,x;\underline{u} (0,\cdot;z,d,d_1,b),z)\ge\underline{u} (t,x;z,d,d_1,b)$
 and  $u(t,x;\bar{u}(0,\cdot;z,d,d_1),z)\le\bar{u}(t,x;z,d,d_1)$ for $t\ge0$.
 For fixed $0<d\le 2$, $d_1>0$ with $\frac{d_1}{d}\gg1$, and $0<b\ll1$, put
  $$\underline{u}(t,x;z)=\underline{u}(t,x;z,d,d_1,b),\quad \bar u(t,x;z)=\bar u(t,x;z,d,d_1).
  $$

\begin{proposition}\label{sub-sup-bound-prop} There is a constant $C$  such that for any $0<b\ll1$ and $\frac{d_1}{d}\gg1$,
\begin{align}\label{sub-sup-bound}
\inf\limits_{x\le C,t\ge0,z\in\R} u(t,x+ct;\bar{u}(0,\cdot;z),z)\ge \inf\limits_{x\le C,t\ge0,z\in\R} u(t,x+ct;\underline{u}(0,\cdot;z),z)>0.
\end{align}
\end{proposition}

\begin{proof}
First, by \eqref{sub-solution}, Propositions \ref{sub-solution-prop} and \ref{sup-solution-prop}, for any $t\ge0$,
\begin{equation}\label{inequality-sup-sub}
\underline{u}(t,x;z)\le u(t,x;\underline{u}(0,\cdot;z),z)\le u(t,x;\bar{u}(0,\cdot;z),z)\le\bar{u}(t,x;z). \end{equation}
Observe that
\begin{align*}
\underline{u}(t,x+ct;z)&=\max\{b\psi^0(t,x+ct+z),\underline{v}(t,x+ct;z,d,d_1)\}\quad\mbox{for }x\le M\\&\ge
b\psi^0(t,x+ct+z)\quad\mbox{for } x\le M\\ &\ge \inf\limits_{t\in\R,x\in\R} b\psi^0(t,x)\\&>0.
\end{align*}
This together with \eqref{inequality-sup-sub} implies \eqref{sub-sup-bound}.
\end{proof}

\begin{lemma}
Let
$$u^n(t,x,z)=u(t+nT,x+cnT;\bar{u}(0,\cdot;z-cnT),z-cnT)$$
and
$$u_n(t,x,z)=u(t+nT,x+cnT;\underline{u}(0,\cdot;z-cnT),z-cnT).$$
Then for any given bounded interval $I\subset\R$, there is $N_0\in\mathbb{N}$ such that $u^n(t,x,z)$ is non-increasing in $n$ and $u_n(t,x,z)$ is non-dereasing in $n$ for $n\ge N_0$, $t\in I$, $x$, $z\in\R$.
\end{lemma}
\begin{proof}
First, observe that
$$\bar{u}(T,x+cT;z-cnT)=\bar{u}(0,x;z-c(n-1)T)\quad\forall\, n\ge0.$$
Hence for given $t\in\R$ and $n\in\mathbb{N}$ with $t+(n-1)T>0$,
\begin{align*}
&u^n(t,x,z)\\&=u(t+nT,x+cnT;\bar{u}(0,\cdot;z-cnT),z-cnT)\\&=u(t+(n-1)T,x+cnT;u(T,\cdot;\bar{u}(0,\cdot;z-cnT),z-cnT),z-cnT)\\&=u(t+(n-1)T,x+c(n-1)T;u(T,\cdot+cT;\bar{u}(0,\cdot;z-cnT),z-cnT),z-c(n-1)T)\\&\le u(t+(n-1)T,x+c(n-1)T;\bar{u}(T,\cdot+cT;z-cnT),z-c(n-1)T)\\&= u(t+(n-1)T,x+c(n-1)T;\bar{u}(0,\cdot;z-c(n-1)T),z-c(n-1)T)\\&=u^{n-1}(t,x,z).
\end{align*}

Similarly, we can prove that for given $t\in\R$ and $n\in\mathbb{N}$ with $t+(n-1)T>0$,
$$u_n(t,x,z)\ge u_{n-1}(t,x,z).$$
\end{proof}

Let $$u^+(t,x,z)=\lim\limits_{n\to\infty}u^n(t,x,z),$$
$$u^-(t,x,z)=\lim\limits_{n\to\infty}u_n(t,x,z),$$
and
$$\Psi^{\pm}_0(x,z)=u^{\pm}(0,x,z).$$

\begin{lemma}
For each $z\in\R$, $u^{\pm}(t,x,z)=u(t,x;\Psi^{\pm}_0(\cdot,z),z)$ for $t\in\R$ and $x\in\R$ and hence $u^{\pm}(t,x,z)$ are entire solutions of \eqref{space-continuous-eqn}.
\end{lemma}
\begin{proof}
We prove the case that $u(t,x,z)=u^+(t,x,z)$. First, note that
\begin{align*}
u^n&(t,x,z)\\=&u(t,x+cnT;u(nT,\cdot;\bar{u}(0,\cdot;z-cnT),z-cnT),z-cnT)\\=&u(t,x;u(nT,\cdot+cnT;\bar{u}(0,\cdot;z-cnT),z-cnT),z)\\
=&u^n(0,x,z)\\&+\int_0^t\big[ H(z)u^n(\tau,x,z)+u^n(\tau,x,z)f(\tau,x+z,u^n(\tau,x,z)) \big]d\tau.
\end{align*}
where $$ H(z)u^n(t,x,z)=d(t,x+z+1)\big(u^n(t,x+1,z)-u^n(t,x,z)\big)+ d(t,x+z-1)\big(u^n(t,x-1,z)-u^n(t,x,z)\big).$$
Then by Lebesgue dominated convergence theorem,
\begin{align*}
u(t,x,z)=&\Psi^+_0(x,z) \\&+\int_0^t\big[ H(z)u(\tau,x,z)+u(\tau,x,z)f(\tau,x+z,u(\tau,x,z)) \big]d\tau.
\end{align*}
This implies that $u(t,x,z)=u(t,x;\Psi^{+}_0(\cdot,z),z)$ for $t\in\R$ and $x\in\R$ and $u(t,x,z)$ is an entire solution of \eqref{space-continuous-eqn}.
\end{proof}

\begin{proof} [Proof of Theorem \ref{main-periodic-thm}]
(1) Note that, following from \cite{LiZh2} and \cite{Wei02}, for any $c\ge c^*$, \eqref{main-eqn} has a periodic traveling wave solution with speed
$c$. But the property \eqref{periodic-wave-eq1} is not established. In the following, we provide a proof of the existence of periodic traveling wave
solutions of \eqref{main-eqn} with speeds $c>c^*$ satisfying the property \eqref{periodic-wave-eq1}, which enables us to use Theorem \ref{main-general-thm} to prove (2) and (3).

Let
$$\Psi^{\pm}(x,t,z)=u^{\pm}(t,x+ct,z-ct)(=u(t,x+ct;\Psi^{\pm}_0(\cdot,z-ct),z-ct)).$$

First of all, $u(t,x;\Psi^{\pm}(\cdot,0,z),z)=\Psi^{\pm}(x-ct,t,z+ct)$ follows directly from the definition of $\Psi^{\pm}(x,t,z)$.

Secondly, we prove that
$$\lim\limits_{x-ct\to\infty}\frac{\Psi^{\pm}(x-ct,t,z+ct)}{ de^{-\mu(x-ct)}\psi^{\mu}(t,x+z)}=1$$ uniformly in $t\in\R$ and $z\in\R$, which is equivalent to
\begin{equation}\label{decay}
\lim\limits_{x\to\infty}\frac{\Psi^{\pm}(x,t,z)}{de^{-\mu x}\psi^{\mu}(t,x+z)}=1\end{equation}
uniformly in $t\in\R$ and $z\in\R$. Note that
\begin{align}\label{inequality-Psi}
\underline{v} (t,x;z,d,d_1)&=de^{-\mu(x-ct)}\psi^\mu(t,x+z)-d_1e^{-\mu^{'}(x-ct)}\psi^{\mu^{'}}(t,x+z)\nonumber\\&\le u(t,x;\Psi^{\pm}(\cdot,0,z),z)\nonumber\\&=\Psi^{\pm}(x-ct,t,z+ct)\nonumber\\&\le \bar{v} (t,x;z,d,d_1)\nonumber\\&=de^{-\mu(x-ct)}\psi^\mu(t,x+z) +d_1e^{-\mu^{'}(x-ct)}\psi^{\mu^{'}}(t,x+z)
\end{align} for $t\in\R$ and $x,z\in\R$. \eqref{decay} then follows from \eqref{inequality-Psi}.

Thirdly, we prove the periodicity of $\Psi^{\pm}(x,t,z)$ in $t$ and $z$. Note that
\begin{equation*}
\Psi^+(x,t,z)=\lim_{n\to\infty}u\big(t+nT,x+cnT+ct;\bar{u}(0,\cdot;z-cnT-ct),z-cnT-ct\big).
\end{equation*}
Then we have
\begin{align}\label{periodicity-time}
\Psi^+(x,T,z)&=\lim_{n\to\infty}u\big((n+1)T,x+c(n+1)T;\bar{u}(0,\cdot;z-c(n+1)T),z-c(n+1)T\big)\nonumber\\&=\lim_{n\to\infty}u\big(nT,x+cnT;\bar{u}(0,\cdot;z-cnT),z-cnT\big)\nonumber\\&=\Psi^+(x,0,z)
\end{align}
and
\begin{align}\label{periodicity-space}
\Psi^+(x,t,z+J)&=\lim_{n\to\infty}u\big(t+nT,x+cnT+ct;\bar{u}(0,\cdot;z+J-cnT-ct),z+J-cnT-ct\big)\nonumber\\&=
\lim_{n\to\infty}u\big(t+nT,x+cnT+ct;\bar{u}(0,\cdot;z-cnT-ct),z-cnT-ct\big)\nonumber\\&=\Psi^+(x,t,z).
\end{align}
Similarly, we have
\begin{equation}\label{periodicity-time2}
\Psi^-(x,T,z)=\Psi^-(x,0,z),
\end{equation}
\begin{equation}\label{periodicity-space2}
\Psi^-(x,t,z+J)=\Psi^-(x,t,z).
\end{equation}

By Proposition \ref{sub-sup-bound-prop},
$$\inf_{x\le C,t\ge0,z\in\R}\Psi^{\pm}(x,t,z)>0.$$
Then by Proposition \ref{positive-entire-solution-prop}, Lemma \ref{convergence-lm} and the periodicity of $\Psi^{\pm}(x,t,z)$ in $t$, we have
\begin{equation}\label{uniform-convergence-proper}
\lim_{x\to -\infty}\big(\Psi^{\pm}(x,t,z)- u^+_{x+z}(t)\big)=0
\end{equation} uniformly in $t\in\R$ and $z\in\R$.

Let $$\Phi^{\pm}(x,t,z)=\Psi^{\pm}(x,t,z-x)\quad\quad\mbox{for }x,\,z\in\R.$$
By \eqref{decay}, \eqref{periodicity-time}-\eqref{uniform-convergence-proper}, $\Phi^{\pm}(x,t,z)$ generate traveling wave solutions with speed $c$ satisfying \eqref{periodic-wave-eq1}.

(2)  It follows from Proposition \ref{sub-super-solution-prop}, \eqref{periodic-wave-eq1} and Theorem \ref{main-general-thm}(1).

(3) It follows from  \eqref{periodic-wave-eq1} and Theorem \ref{main-general-thm}(2).
\end{proof}

\section{Existence, stability and uniqueness of transition waves in time heterogeneous media}

In this section, we assume that $d(t,j)\equiv d(t)$ and $f(t,j,u)\equiv f(t,u)$, and study the existence, uniqueness, and stability of
transition waves of \eqref{main-eqn}.

We first recall some results on  transition waves established in the recent paper \cite{CaSh}.
Define
$$
\bar f_{\inf}=\liminf_{t\ge s, t-s\to\infty}\frac{1}{t-s}\int_s^t f(\tau,0)d\tau,
$$
$$
\bar f_{\sup}=\limsup_{t\ge s, t-s\to\infty}\frac{1}{t-s}\int_s^t f(\tau,0)d\tau,
$$
$$
\bar f_{\inf}^+=\liminf_{t\ge s\ge 0, t-s\to\infty}\frac{1}{t-s}\int_s^t f(\tau,0)d\tau,
$$
and
$$
\bar f_{\sup}^+=\limsup_{t\ge s\ge 0, t-s\to\infty}\frac{1}{t-s}\int_s^t f(\tau,0)d\tau.
$$

For given $\mu>0$, let
$$c(t;\mu)=\frac{e^{-\mu}+e^{\mu}-2+f(t,0)}{\mu}.
$$
Let
$$
 \tilde c_0^-:=\inf \limits_{\mu>0}\frac{e^{-\mu}+e^\mu-2+\bar f_{\inf}}{\mu}.
 $$
 By \cite[Lemma 5.1]{CaSh},
 There is a unique $\mu^*>0$ such that
$$
\tilde c_0^-=\frac{e^{-\mu^*}+e^{\mu^*}-2+\bar f_{\inf}}{\mu^*}
$$
and
for any $\gamma>\tilde c_0^-$,
 the equation $\gamma=\frac{e^{-\mu}+e^{\mu}-2+\bar f_{\inf}}{\mu}$ has exactly two positive solutions for $\mu$.

For any $\gamma>\tilde c_0^-$, let $0<\mu<\mu^*$ be such that $\frac{e^{-\mu}+e^{\mu}-2+\bar f_{\inf}}{\mu}=\gamma$ and $c(t)=c(t;\mu)$.  Let
$$
\phi(t,j)=e^{-\mu(j-\int_0^ t c(\tau)d\tau)}.
$$
Let  $B(t)=-(e^{-\tilde{\mu}}+e^{\tilde{\mu}}-2)+c(t)\tilde{\mu}-f(t,0)$.
Note that
\begin{eqnarray*}
\bar B_{\inf}&=&-(e^{-\tilde{\mu}}+e^{\tilde{\mu}}-2)+\gamma \tilde{\mu}-\bar f_{\inf}\\&=&\tilde{\mu}(\gamma-\frac{e^{-\tilde{\mu}}+e^{\tilde{\mu}}-2+\bar f_{\inf}}{\tilde{\mu}}),
\end{eqnarray*}
thus we can choose $\tilde{\mu}\in(\mu,2\mu)$ such that $\bar B_{\inf}>0$.
By \cite[Lemma 3.2]{NaRo12}, there is $A\in W^{1,\infty}(\RR)$ such that ${\rm essinf}_{t\in\RR}(A^{\prime}+B)>0$.
Let
$$
\phi_1(t,j)=e^{A(t)-\tilde\mu (j-\int_0^t c(\tau)d\tau)}.
$$

\begin{proposition}
\label{sub-super-time-dependent-prop}
For given $u^0\in l^{\infty,+}$ and $t_0\in\RR$, if
$$
u^0\le d\phi(t_0,\cdot)+d_1 \phi_1(t_0,\cdot)\quad ({\rm resp.}, \,\, u^0\ge d\phi(t_0,\cdot)-d_1\phi_1(t_0,\cdot))
$$ for $0<d\le 2$ and $d_1\gg 1$, then
$$
(t;t_0,u^0)\le d\phi(t,\cdot)+d_1 \phi_1(t,\cdot)\quad ({\rm resp.}, \,\, u(t;t_0,u^0)\ge d\phi(t,\cdot)-d_1\phi_1(t,\cdot))
$$
for all $t\ge t_0$.
\end{proposition}

\begin{proof}
It follows from the arguments of \cite[Lemma 5.2]{CaSh} and $f_u(t,u)<0$ for $u\ge0$.
\end{proof}

 \begin{proposition}
 \label{existence-time-dependent-prop}
  For any $\gamma>\tilde c_0^-$,  let $0<\mu<\mu^*$ and  $c(t)=\frac{e^{-\mu}+e^{\mu}-2+f(t,0)}{\mu}$ be such that $\bar c_{\inf}=\gamma$. Then there exists a transition wave solution $u_j(t)=U(t,j)$ satisfying that
  $$
  \phi(t,j)-d_1^* \phi_1(t,j)\le U(t,j)\le \phi(t,j)+d_1^* \phi_1(t,j)
  $$
  for some $d_1^*>0$.
 \end{proposition}

\begin{proof}
It follows from the arguments of \cite[Theorem 1.3]{CaSh}.
\end{proof}

\begin{theorem}
\label{main-time-dependent-thm}  For any $\gamma>\tilde c_0^-$,  let $0<\mu<\mu^*$ and  $c(t)=\frac{e^{-\mu}+e^{\mu}-2+f(t,0)}{\mu}$ be such that $\bar c_{\inf}=\gamma$. Let $u_j(t)=U(t,j)$ be the transition wave solution in Proposition \ref{existence-time-dependent-prop}.

\begin{itemize}
\item[(1)] (Stability)  For any $u^0\in l^{\infty,+}$ and $t_0\in\RR$ satisfying  that
$$
\inf_{j\le j_0} u^0_j>0\quad \forall\,\, j_0\in\Z,\quad \lim_{j\to\infty} \frac{u^0_j}{U(t_0,j)}=1,
$$
there holds
$$
\lim_{t\to\infty} \frac{u_j(t+t_0;t_0,u^0)}{U(t+t_0,j)}=1
$$
uniformly in $j\in\Z$.

\item[(2)] (Uniqueness) If $u_j(t)=V(t,j)$ is  a transition wave solution of \eqref{main-eqn} satisfying that
$$
\lim_{j\to\infty} \frac{V(t,j+[\int_0^ t c(\tau)d\tau])}{U(t,j+[\int_0^t c(\tau)d\tau])}=1
$$
uniformly in $t\in\RR$, then
$$
V(t,j)\equiv U(t,j).
$$
\end{itemize}
\end{theorem}

\begin{proof}
(1) It follows from Propositions \ref{sub-super-time-dependent-prop} and \ref{existence-time-dependent-prop}, and Theorem \ref{main-general-thm}(1).

(2) It follows from Propositions \ref{sub-super-time-dependent-prop} and \ref{existence-time-dependent-prop}, and Theorem \ref{main-general-thm}(2).
\end{proof}

\end{document}